\documentclass[letterpaper, 10 pt, conference]{ieeeconf}  % Comment this line out if you need a4paper

\IEEEoverridecommandlockouts                              % This command is only needed if 
\usepackage[T1]{fontenc}
\usepackage{graphicx} % for pdf, bitmapped graphics files
\usepackage{amsmath} % assumes amsmath package installed
\usepackage{amssymb}  % assumes amsmath package installed
\usepackage{mathtools}
\usepackage[hyphens]{url}
\usepackage[noadjust]{cite}
\usepackage{bbm}

\usepackage[ruled,vlined,linesnumbered]{algorithm2e}
\usepackage{algorithmic,algorithm2e,float}

\SetKw{Continue}{continue}

\usepackage{amsthm}
\theoremstyle{definition}
\newtheorem{problem}{Problem}
\newtheorem{thm}{Theorem}

\newtheorem{definition}{Definition}

\newtheorem{observation}{Observation}

\newtheorem*{rem}{Remark}

\newtheorem{lem}{Lemma}

\newcommand\oprocendsymbol{\hbox{$\bullet$}}
\newcommand\oprocend{\relax\ifmmode\else\unskip\hfill\fi\oprocendsymbol}

\title{\LARGE \bf %A Unified Approach to Sensor Scheduling and Sensor Selection Problems for Optimal Kalman Filtering  
A Unified Approach to Optimally Solving Sensor Scheduling \\ and Sensor Selection Problems in Kalman Filtering
}

\author{Shamak Dutta, Nils Wilde, and Stephen L. Smith% <-this % stops a space
\thanks{This research is supported in part by the Natural Sciences and Engineering
Research Council of Canada (NSERC) and by Nutrien Ltd.}
\thanks{S.\ Dutta and S.\ L.\ Smith are with the Department of Electrical and Computer Engineering,
        University of Waterloo, Canada 
        \{{\tt\small stephen.smith, shamak.dutta\}@uwaterloo.ca}.  N.\ Wilde is with the Cognitive Robotics Department, Delft University of Technology, Netherlands ({\tt\small N.Wilde@tudelft.nl}).}%
}
\begin{document}
\maketitle
\thispagestyle{empty}
\pagestyle{empty}

\begin{abstract}
We consider a general form of the sensor scheduling problem for state estimation of linear dynamical systems, which involves selecting sensors that minimize the trace of the Kalman filter error covariance (weighted by a positive semidefinite matrix) subject to polyhedral constraints on the selected sensors. This general form captures several well-studied problems including sensor placement, sensor scheduling with budget constraints, and Linear Quadratic Gaussian (LQG) control and sensing co-design. We present a mixed integer optimization approach that is derived by exploiting the optimality of the Kalman filter. While existing work has focused on approximate methods to specific problem variants, our work provides a unified approach to computing optimal solutions to the general version of sensor scheduling. In simulation, we show this approach finds optimal solutions for systems with 30 to 50 states in seconds.
\end{abstract}

% ----------------------------------

\section{Introduction} \label{section:introduction}

In many large scale dynamical systems, it is practically infeasible to measure \emph{all} output variables due to communication or energy constraints. Thus, practical solutions often have to rely on only accessing a subset of the data. Optimizing the subset of output variables to measure in order to best perform state estimation is commonly known as the \emph{sensor scheduling problem}. This finds applications in multi-robot environmental monitoring \cite{le2009trajectory,lan2016rapidly,placed2023survey}, minimal controllability problems \cite{olshevsky2014minimal,tzoumas2015minimal,summers2015submodularity}, control and sensing co-design \cite{tzoumas2020lqg,zardini2021co,siami2020separation}, among others.

Sensor scheduling considers a stochastic linear system evolving in discrete time, which can be partially observed using sensors. Due to resource constraints (e.g., energy constraints), at each time step only a subset of sensors can be turned on. By limiting the number of active sensors, the sensor battery life is prolonged, allowing for long-term usage. The objective is to determine a sensor schedule, i.e., a subset of sensors to  be activated at each step, that minimizes the trace of the Kalman filter weighted error covariance while satisfying the resource constraints. The weight is provided by a positive semidefinite matrix which can encode several objectives including final state error and average error. When the chosen sensor set is not allowed to change over time, the problem is known as the \emph{sensor selection} problem. The sensor scheduling problem and its variants are challenging; existing approaches have provided approximate solutions via greedy algorithms \cite{shamaiah2010greedy,jawaid2015submodularity,tzoumas2016sensor,tzoumas2016near, zhang2017sensor,singh2017supermodular,ye2020complexity,chamon2020approximate,kohara2020sensor,tzoumas2020lqg}, convex relaxations \cite{joshi2008sensor,weimer2008relaxation,moshtagh2009optimal,mo2011sensor,wang2013convex,carmi2013sensor,munz2014sensor,dhingra2014admm,maity2022sensor}, and randomized algorithms \cite{siami2020deterministic,hashemi2020randomized,calle2021probabilistic}. However, in this paper we strive to compute \emph{optimal} solutions leveraging advances in mixed integer optimization.

While approximation algorithms have yielded useful insights into specific variants of the problem, an approach that computes optimal solutions to the general version has been elusive. The main difficulty lies in the analysis of the Kalman filter recursive equations in terms of the selected sensor subsets. However, the combinatorial nature of the problem suggests encoding the problem as an integer program.

Despite integer programming being NP-complete \cite{dasgupta2008algorithms} and generally being considered intractable beyond specific cases, the mixed integer optimization community has made tremendous theoretical and practical advances in last few decades with \emph{machine-independent} speedup factors of upto 580,000 \cite[Section 2.1]{bertsimas2016best}. One of the benefits is that modern solvers are anytime i.e., if terminated early, approximate solutions with suboptimality certificates are provided to the user. It is relatively straightforward to set up most combinatorial problems as integer programs by modelling the objects of interest as binary decision variables. However, the structure of the resulting objective may not be amenable to optimization. It is well known that the formulation plays an important role in solving integer programs \cite[Section 1.2]{gunluk2011perspective}. This is because the relaxation is used extensively to provide lower bounds in the branch and bound procedure in mixed integer optimization. Thus, if the relaxation has good structure (linearity, convexity) and is tight, one can find optimal solutions relatively quickly by pruning many parts of the search tree. The question we wish to answer is whether one can leverage the structure of the sensor scheduling problem to formulate mixed integer programs.

\emph{Contributions:}
Our main contribution is the formulation of the general sensor scheduling problem as a mixed integer program with a convex quadratic objective and linear constraints (Theorem~\ref{thm:miqp}). This enables us to leverage modern optimization solvers such as \textsc{Gurobi} \cite{gurobi} or \textsc{CPLEX} \cite{cplex2009v12} to compute optimal solutions to sensor scheduling. We accomplish this using the property of the optimality of the Kalman filter for the trace of a weighted error covariance. This enables us to optimize over the class of linear filters, which results in a convex minimization problem. This circumvents the need to work with the recursive form of the Kalman filter and the associated non-linear algebraic Riccati equations, which can be difficult to model. In simulations, we show that our approach is capable of finding optimal solutions to sensor selection and scheduling problems on systems with 35 to 50 states in seconds.

\emph{Related Work:}
Sensor scheduling has been studied in various forms in the literature. The difference in formulations generally boils down to the choice of objective function, defined in terms of the Kalman filter error covariance. The work in \cite{jawaid2015submodularity} characterizes the submodularity of commonly used functions such as the trace, maximum eigenvalue, and log determinant. This paper considers the trace as it is a direct measure of the estimation error which we seek to minimize. 

The trace is not a submodular function and does not benefit from the constant factor approximation guarantees provided by the greedy algorithm. The typical way to address this is to use surrogate objectives such as the log determinant, which is submodular and can be efficiently approximately maximized. However, \cite[Remark 1]{chamon2020approximate} notes that it is a poor proxy for minimizing the trace. There is a large body of work tackling the log determinant/maximum eigenvalue objective \cite{shamaiah2010greedy,jawaid2015submodularity,tzoumas2016near,tzoumas2016sensor,calle2021probabilistic,joshi2008sensor,moshtagh2009optimal} and various other controllability metrics \cite{siami2020deterministic,wang2013convex,dhingra2014admm}. Unfortunately, none of the aforementioned algorithms provide optimal solutions to trace minimization over a finite horizon.

The trace minimization of the \emph{steady state} error covariance has been considered \cite{gupta2006stochastic,zhao2014optimal,zhang2017sensor,ye2020complexity}. It was established to be NP-hard in \cite{zhang2017sensor}, which also demonstrated that greedy algorithms perform well in practice but without guarantees. In fact, there is no polynomial time constant factor approximation algorithm for this problem unless P $=$ NP \cite{ye2020complexity}.

On a \emph{finite horizon}, greedy algorithms are a popular tool for minimizing the trace of the error covariance due to its computational efficiency and empirical success. Recent work has analyzed their performance through the lens of approximate submodularity \cite{chamon2020approximate}, weak submodularity \cite{hashemi2020randomized,kohara2020sensor}, and identified strict conditions for submodularity \cite{singh2017supermodular}. However, \cite[Section VI]{chamon2020approximate} notes that the performance bounds tend to be quite loose in practice and greedy algorithms perform near-optimally on small scale instances. The work in this paper computes optimal solutions to instances not previously considered which may help in a more thorough empirical study of the greedy algorithm.

There has been limited work in computing the optimal solution to the sensor scheduling problem. Convex relaxations have provided useful insights \cite{maity2022sensor,weimer2008relaxation,wang2013convex,carmi2013sensor,dhingra2014admm} for approximate solutions. The work in \cite{vitus2012efficient} characterizes certain properties of the optimal solution and uses it to construct a suboptimal solution via tree pruning. The work closest to ours is \cite{mo2011sensor}, which formulates a quadratic program with $\ell_0$ constraints. The authors solve a relaxed version of the problem and construct an approximate solution through iterative reweighting. As a baseline comparison, we convert the $\ell_0$ constraint to an indicator constraint to form a MIQP, which we discuss in detail in Section~\ref{rem:comparison}.

\emph{Outline}: We begin with a brief introduction to linear estimation in Section~\ref{section:preliminaries}. This contains results about the optimality of linear estimators which apply to Kalman filtering. Section~\ref{section:problem_formulation} contains the problem formulation as well as a few examples of the problem types it captures. Our reformulation of sensor scheduling as a mixed integer optimization problem is presented in Section~\ref{section:mixed_integer_formulation}. We investigate its performance through numerical simulations in Section~\ref{section:results} and conclude with some directions for future work in Section~\ref{section:conclusions}.
\section{Preliminaries} \label{section:preliminaries}

    \emph{Notation}: Let $[n]$ denote the set of positive integers ranging from 1 to $n$. For a vector $v \in \mathbb{R}^n$, define $\text{supp}(v) := \{i \in [n]: v_i \neq 0\}$. Further, for any $S \subseteq [n]$, let $v_S \in \mathbb{R}^{|S|}$ be the vector with entries $v_i$ for $i \in S$. The set of $n \times n$ positive-definite matrices is denoted by $\mathbb{S}^n_{++}$. For a matrix $X \in \mathbb{R}^{m \times n}$, denote the Frobenius norm by $\lVert X \rVert_{F} := \left(\mathrm{trace}(X'X)\right)^{1/2}$. For two random vectors $x, y$ with dimensions $n$ and $m$ respectively, denote the matrix of covariances between $x$ and $y$ by $\Sigma_{xy} := \mathbb{E}[(x - \mathbb{E}[x]) (y - \mathbb{E}[y])'] \in \mathbb{R}^{n \times m}$. 
        \subsection{Linear Estimation} \label{subsec:llse}
         Let $x$ and $y$ be jointly distributed random vectors (of size $n$ and $m$ respectively) with zero mean and covariance given by
         \begin{equation}
             \begin{bmatrix}
                \Sigma_{xx} & \Sigma_{xy}\\
                \Sigma_{yx} & \Sigma_{yy} 
             \end{bmatrix} \in \mathbb{R}^{(n+m) \times (n+m)}.
         \end{equation}
         The following characterizes the optimal linear estimator.
        \begin{definition}[Theorem 2.1, \cite{anderson2012optimal}] \label{definition:linear_least_squares_estimator}
        The optimal linear estimator of $x$ given $y$ is defined as
        \begin{equation}
        \hat{x} := {K}_* {y},
        \end{equation}
        where the optimal coefficient matrix ${K}_* \in \mathbb{R}^{n \times m}$ is the solution that minimizes the expected squared error
        \begin{equation}
        \label{eq:Kstar}
            \begin{split}
                K_* &:= \underset{{K}}{\text{arg min }} \mathbb{E} \Big[ \left( x - {K} {y} \right)' ( x - Ky ) \Big] = \Sigma_{xy} \Sigma_{yy}^{-1}.
            \end{split}
        \end{equation}
        \end{definition}
    The resulting estimation error is
    \begin{equation}
        \mathbb{E} \left[ \left(x - K_* y\right)' \left(x - K_* y\right) \right] = \mathrm{trace} \left(\Sigma_{xx} - \Sigma_{xy} \Sigma_{yy}^{-1} \Sigma_{yx}\right).
    \end{equation}
    The optimal linear estimator also satisfies a stronger property, given in the following lemma.  This is a relatively straightforward extension of the previous result, and we include the proof for completeness.
    \begin{lem} \label{lem:linear-estimator-optimality}
        Given a positive semi-definite matrix $M = Q'Q \in \mathbb{S}^{n}_{+}$, the coefficient matrix $K_*$ in~\eqref{eq:Kstar} is also minimizes the $M$-weighted expected square error:
        \begin{equation} \label{eqn:optimality-linear}
            K_{*} = \underset{{K} \in \mathbb{R}^{n \times m}}{\text{arg min }} \mathbb{E}\Big[ (x - Ky)' M ( x - Ky) \Big].
        \end{equation}
    \end{lem}
    \begin{proof}
        Consider the trace of scalar term in the minimization,
        \begin{equation}
            \begin{split}
                \mathrm{trace} &\Big( \mathbb{E}\left[ (x - Ky)' Q' Q (x - Ky)\right] \Big) = \\
                & \mathrm{trace} \Big( Q \mathbb{E}\left[ (x - Ky) (x - Ky)' \right] Q' \Big) \\
                & =\mathrm{trace} \Big(Q \Big(\Sigma_{xx} - K \Sigma_{yx} - \Sigma_{xy} K' + K \Sigma_{yy} K' \Big) Q'\Big)\\
                &= \mathrm{trace} \Bigg( Q \Big(\Sigma_{xx} -  \Sigma_{xy}\Sigma_{yy}^{-1}\Sigma_{yx}\Big) Q'\\
                &+ Q \Big(K - \Sigma_{xy} \Sigma_{yy}^{-1}\Big) \Sigma_{yy} 
                \Big(K' - \Sigma_{yy}^{-1} \Sigma_{yx} \Big) Q' \Bigg).
            \end{split}
        \end{equation}
        The first term does not depend on $K$ and the second term is non-negative since it is equal to $\lVert \Sigma_{yy}^{1/2}(K' - \Sigma_{yy}^{-1} \Sigma_{yx}) Q'\rVert_{F}$. The second term can be made zero by setting $K = \Sigma_{xy} \Sigma_{yy}^{-1}$ which is the coefficient of the optimal linear estimator.
    \end{proof}

\section{Problem Formulation} \label{section:problem_formulation}

Our setup closely follows \cite{mo2011sensor} which captures a variety of interesting sensor scheduling problems. 

\emph{Dynamical System}: Consider the following linear system:
    \begin{equation} \label{eqn:dynamics}
        \begin{split}
            x_{k+1} &= A x_k + w_k,\\
            y_{k} &= C x_k + v_k,
        \end{split}
    \end{equation}
where the process noise $w_k$, measurement noise $v_k$, and initial state $x_1$ are zero mean uncorrelated random variables. For all $k \geq 1$, we assume the covariance of $w_k$ and $v_k$ are given by $W \in \mathbb{S}^{n}_{++}$ and $V \in \mathbb{S}^{m}_{++}$ respectively and the \emph{a priori} covariance of the initial state $x_1$ is given by $\Sigma_{1|0} \in \mathbb{S}^n_{++}$. Further, we assume the state $x_k \in \mathbb{R}^n$ and the vector resulting from all $m$ sensor measurements is $y_k \in \mathbb{R}^m$.

\emph{Sensor Schedule}: For a given time horizon $T$, the total number of measurements is $mT$ since each step yields $m$ sensor readings. For $i \in [mT]$, consider binary indicator variables $\gamma_i \in \{0,1\}$ such that if sensor $j \in [m]$ is on at time step $k \in [T]$, then $\gamma_{m(k-1)+j}=1$. The binary vector ${\gamma} = [\gamma_1, \ldots, \gamma_{mT}] \in \{0,1\}^{mT}$ is called a \emph{sensor schedule}.

\emph{Kalman filter}: The Kalman filter (KF) is usually described recursively in terms of the estimates from the previous time step. However, viewing it directly as a linear estimator over the set of sensor measurements will be more useful for our purposes. Let $\gamma \in \{0,1\}^{mT}$ be a sensor schedule over a horizon $T$ with $S_{\gamma} := \text{supp}(\gamma)$. For each $k \in [T]$, define
    \begin{equation} \label{eqn:kf-defns}
        \begin{split}
            {Y}_k &:= \Big[y_{1}', \ldots, y_{k}'\Big]' \in \mathbb{R}^{mk} \; (\text{\footnotesize{measurements until time $k$}})\\
            \hat{x}_{k|k}^{\gamma} &:= K_{*,k}^{\gamma} Y_{k, \gamma} \in \mathbb{R}^n \; (\text{\footnotesize{KF estimate using selected measurements}})\\
            e_k^{\gamma} &:= x_k - \hat{x}_{k|k}^{\gamma} \in \mathbb{R}^n \; (\text{\footnotesize{KF estimation error}})\\
            \Sigma^{\gamma}_{k|k} &:= \Sigma_{e_k^{\gamma} e_k^{\gamma}} \in \mathbb{S}^{n}_{+} \; (\text{\footnotesize{KF posterior error covariance}})
        \end{split}
    \end{equation}
Note that $Y_{k, \gamma} \in \mathbb{R}^{|S_{\gamma}|}$ only uses the measurement variables selected by the sensor schedule $\gamma$. Further, $K^{\gamma}_{*,k} \in \mathbb{R}^{n \times |S_{\gamma}|}$ contains the optimal Kalman filter coefficients for estimating the state $x_k$ using measurement variables $Y_{k, \gamma}$. This is different from the gain matrix used in the recursive form of the filter. Each row $i \in [n]$ of $K^{\gamma}_{*,k}$ contains the coefficients of the optimal linear combination of the measurement variables used in estimating the $i^{\text{th}}$ entry of the state $x_k$.

In summary, given the schedule $\gamma$, the Kalman filter computes the optimal linear estimate $\hat{x}^{\gamma}_{k|k}$ of the state $x_k$ and provides the \emph{a posteriori} error covariance estimate $\Sigma^{\gamma}_{k|k}$. Our goal is to compute the schedule $\gamma$ that minimizes a function related to $\Sigma_{k|k}^{\gamma}$, subject to resource constraints on $\gamma$.
\begin{problem}[Sensor Scheduling] \label{problem:1}
 Given a time horizon $T > 0$, cost matrices $Q_k \in \mathbb{R}^{n' \times n}$ with $M_k := Q_k'Q_k \in \mathbb{S}^{n}_{+}$, constraint matrix $H \in \mathbb{R}^{h \times mT}$, and budget $b \in \mathbb{R}^h$, compute the sensor schedule $\gamma$ that solves
    \begin{equation} \label{eqn:problem}
        \begin{aligned}
        & \underset{\gamma \in \{0,1\}^{mT}}{\text{minimize}}
        & & \sum_{k=1}^T \mathrm{trace}\left( Q_k \Sigma_{k|k}^{\gamma} Q_k' \right)\\
        & \text{subject to}
        & & H \gamma \leq b.
        \end{aligned}
    \end{equation}
\end{problem}
The objective and constraints capture several interesting problems. For example, consider the following objectives.
\begin{enumerate}
    \item \emph{Total Error}: $Q_k = \theta_k \mathbb{I}_n, k \in [T]$, where $\theta_k \in \mathbb{R}_{\geq 0}$.
    \item \emph{Final State Error}: $Q_1 = \ldots = Q_{T-1} = 0, Q_T = \mathbb{I}_n$.
    \item \emph{LQG Sensing Design}: It is known that the optimal strategy for LQG control is to provide state estimates via Kalman filtering and then perform LQR control using the state estimate in place of the true state. A similar principle holds in LQG control and sensing co-design \cite[Theorem 1]{tzoumas2020lqg} where one designs a sensor selection policy as well as the optimal controller. The result states the budgeted sensor selection problem can be decoupled from the control design where the sensor schedule is computed to optimize $\mathrm{trace}(\Theta \Sigma_{k|k}^{\gamma})$, and $\Theta$ is computed from the LQG system parameters.
\end{enumerate}
Further, the polyhedral constraints in $H \gamma \leq b$ are general and capture different types of resource constraints.
\begin{enumerate}
    \item \emph{Sensor Selection}: One commits to a sensor subset of size $p$ for all steps: $\sum_{i=1}^m \gamma_i = p$ and $\gamma_i = \gamma_{i+jm}$, for $i \in [m], j \in [T-1]$ effectively using only $m$ variables.
    \item \emph{Sensor Scheduling}: One selects a sensor subset for each time step $k \in [T]$ of size $p$: $\sum_{i=1}^{m} \gamma_{(k-1)m + i} = p$. 
    \item \emph{Energy Constraints}: If sensor $i$ incurs an energy cost $\alpha_i$ each time it is switched on and has an energy budget $\beta_i$, then the constraint can be represented as $\sum_{k=1}^T \gamma_{(k-1)m+i} \leq \beta_i / \alpha_i$.
\end{enumerate}
Various problems in the literature are modeled as combinations of these objectives and constraints. For example, objective 1) and constraint 1) model the \emph{sensor selection} problem \cite{weimer2008relaxation,carmi2013sensor, jawaid2015submodularity,tzoumas2016sensor,singh2017supermodular,chamon2020approximate,kohara2020sensor,hashemi2020randomized}. Combining the accumulated error in objective 2) and constraint 2) is sensor scheduling with budget constraints \cite{vitus2012efficient,maity2022sensor}.
\section{Mixed Integer Program for Sensor Scheduling} \label{section:mixed_integer_formulation}

    This section describes our MIQP formulation. As discussed in the introduction, one requires good formulations in mixed integer programming to compute optimal solutions relatively quickly \cite{gunluk2011perspective}. For example, if the objective is convex (when the binary variables are relaxed), one can quickly solve the ensuing convex optimization problem to get a lower bound on the optimal solution to the mixed integer program. This helps in eliminating parts of the search tree by comparing lower bounds (by solving the relaxed problem) and upper bounds given by feasible solutions. In the context of sensor scheduling, it is not immediately obvious how one would model the objective in \eqref{eqn:problem} as a convex function.

    Our main contribution is a MIQP formulation whose objective is convex in the continuous variables. This is presented in Section~\ref{subsec:binary-convex-reform} with Theorem~\ref{thm:miqp} giving the main result. Note that the results in Section~\ref{subsec:binary-convex-reform} depend on the covariance matrices between the states and observations in system \eqref{eqn:dynamics}, which are fully specified in Section~\ref{subsec:covariance-matrices}.
    
    \subsection{Binary Convex Reformulation} \label{subsec:binary-convex-reform}
    We wish to rewrite the objective in \eqref{eqn:problem} purely in terms of continuous variables. By exploiting the optimality of the Kalman filter, we reformulate the summands in \eqref{eqn:problem} as convex minimization problems where the continuous variables are the coefficients of linear filters. The main idea is to optimize over the class of linear filters using \emph{all} sensor measurements with the constraint that a measurement variable $i$ can only be used if its corresponding entry in the schedule satisfies $\gamma_i = 1$. This is established in Lemma~\ref{lem:reformulation}. 

    \begin{lem}[Reformulation of \eqref{eqn:problem}] \label{lem:reformulation}
        Given a schedule $\gamma \in \{0,1\}^{mT}$ and a time step $k \in [T]$, each summand in \eqref{eqn:problem} can be represented as the following optimization problem
        \begin{equation} \label{eqn:lem2}
            \begin{split}
                \mathrm{trace} \left( Q_k \Sigma_{k|k}^{\gamma} Q_k' \right) = \min \Big\{ c_k(K)&: [K]_{ij} ( 1-\gamma_j) = 0, \\
                &i \in [n], j \in [mk],\\
                &K \in \mathbb{R}^{n \times mk}\Big\},
            \end{split}
        \end{equation}
        where $c_k: \mathbb{R}^{n \times mk} \rightarrow \mathbb{R}_{\geq 0}$ is the expected squared error of a linear filter specified by coefficients $K$ and is defined as
        \begin{equation} \label{eqn:lem2-defn}
            \begin{split}
                c_k(K) &:= \mathbb{E} \Big[ \left(x_k - K Y_k\right)' M_k \left( x_k - K Y_k\right)\Big]\\
                &= \mathrm{trace} \Big(M_k \Big( K \Sigma_{Y_k Y_k} K' - 2 \Sigma_{x_k Y_k} K' + \Sigma_{x_k x_k} \Big)\Big).
            \end{split}
        \end{equation}
    \end{lem}
    \begin{proof}
    Let $S_{\gamma} := \text{supp}(\gamma) \cap [mk]$. Consider the RHS in \eqref{eqn:lem2}. The constraint gives the following equivalence: $\gamma_j = 0 \implies K_j = 0$ i.e., the $j^{\text{th}}$ column of $K$ is the zero vector. This implies that when taking a linear combination $K Y_k$, the measurements in $Y_k$ corresponding to the zero column vectors in $K$ can be dropped. Specifically, let $K^\gamma \in \mathbb{R}^{n \times |S_{\gamma}|}$ be the submatrix of $K$ with column set $ {S_{\gamma}}$ and let $Y_{k, \gamma} \in \mathbb{R}^{|S_{\gamma}|}$ be the measurement vector constructed by selecting the entries corresponding to ${S_{\gamma}}$ in $Y_k$. Then, we can rewrite the minimization problem in \eqref{eqn:lem2} as
        \begin{equation} \label{eqn:lem2-reform}
            \begin{split}
                \underset{K \in \mathbb{R}^{n \times |S_{\gamma}|}}{\min} \Bigg\{ \mathbb{E} \Big[ (x_k - K^{\gamma} Y_{k, \gamma})' M_k (x_k - K^{\gamma} Y_{k, \gamma}) \Big] \Bigg\}.
            \end{split}
        \end{equation}
    Now, we can apply Lemma~\ref{lem:linear-estimator-optimality} which tells us the solution to the above minimization problem is the optimal linear estimator of $x_k$ given $Y_{k, \gamma}$ i.e., the Kalman filter with coefficient matrix $K^{\gamma}_{*,k} \in \mathbb{R}^{n \times |S_{\gamma}|}$. Using the definitions in \eqref{eqn:kf-defns}, the resulting error is obtained by plugging in $K^{\gamma}_{*,k}$ and taking the trace of the scalar term in \eqref{eqn:lem2-reform} yielding
    \begin{equation}
        \begin{split}
             &\mathbb{E} \Big[ (x_k - K^{\gamma}_{*, k} Y_{k, \gamma})' Q_k' Q_k (x_k - K^{\gamma}_{*, k} Y_{k, \gamma}) \Big] \\
             &= \mathrm{trace} \Big( Q_k \Sigma^{\gamma}_{k|k} Q_k'\Big),
        \end{split}
    \end{equation}
    where we have used the property $\mathrm{trace}(AB) = \mathrm{trace}(BA)$ for any conformable matrices $A, B$.
    \end{proof}

    Note that we delay the definition of the covariance matrices appearing in $c_k$ until Section~\ref{subsec:covariance-matrices} since it is not crucial to understanding the formulated MIQP. The interested reader can refer directly to Equation~\eqref{eqn:matrix-defns} for the exact definition of the covariance matrices. Next, we show that $c_k$ is convex.
    \begin{lem}[Convexity] \label{lem:convexity}
    The function $c_k: \mathbb{R}^{n \times mk} \rightarrow \mathbb{R}_{\geq 0}$ defined in Lemma~\ref{lem:reformulation} is convex.
    \end{lem}
    \begin{proof}
        To prove convexity, define the following matrix functions $g: \mathbb{S}^{n} \rightarrow \mathbb{R}$ and $h: \mathbb{R}^{n \times mk} \rightarrow \mathbb{S}^n$,
        \begin{equation}
            \begin{split}
                g_k(X) &:= \mathrm{trace}(M_k X)\\
                h(K) &= K \Sigma_{Y_k Y_k} K' - \Sigma_{x_k Y_k} K' - K \Sigma_{Y_k x_k} + \Sigma_{x_k x_k}.
            \end{split}
        \end{equation}
        Then, we have $c_k = g_k \circ h$. Using a composition theorem for convexity, $c_k$ is convex if $g$ is convex and non-decreasing and $h$ is convex \cite[Section 3.6.2]{boyd2004convex}. Convexity of $g$ follows from the linearity of the trace operator and since $M_k \in \mathbb{S}^{n}_{+}$, it is non-decreasing \cite[Example 3.46]{boyd2004convex}. Since $\Sigma_{Y_k Y_k} \in \mathbb{S}^{m}_{+}$, $h$ is also convex \cite[Example 3.49]{boyd2004convex}.
    \end{proof}

    Finally, we establish that the sensor scheduling problem \eqref{eqn:problem} can be represented as a mixed integer convex quadratic program, which can be tackled by solvers like \textsc{Gurobi} \cite{gurobi}.
    \begin{thm}[MIQP for Sensor Scheduling] \label{thm:miqp}
    The sensor scheduling problem~\eqref{eqn:problem} can be formulated as the following mixed-integer convex optimization problem
        \begin{equation} \label{eqn:miqp}
            \begin{aligned}
                & \underset{\substack{\gamma \in \{0,1\}^{mT}\\ K_k \in \mathbb{R}^{n \times mk}}}{\mathrm{min}}
                & & \sum_{k=1}^T c_k(K_k)\\
                & \text{subject to}
                & & H \gamma \leq b, \\
                &&& [K_{k}]_{ij}\left( 1-\gamma_j \right) = 0, i \in [n], \; j \in [mk], \; k \in [T]
            \end{aligned} 
        \end{equation}
where the functions $c_k(\cdot)$ are given by Lemma~\ref{lem:reformulation}.
    \end{thm}
    \begin{proof}
    Since each state vector $x_k$ requires a set of coefficients $K_k \in \mathbb{R}^{n \times mk}$, the result follows from the problem definition \eqref{eqn:problem} and Lemma~\ref{lem:reformulation}. Since convexity is preserved under addition, the objective in \eqref{eqn:miqp} is convex by Lemma~\ref{lem:convexity}.
    \end{proof}

    The optimization problem \eqref{eqn:miqp} is a MIQP of the sensor scheduling problem where the decision variables are the filter coefficients $K_k$ and binary sensor schedule $\gamma$. The last constraint allows a filter coefficient to be used only if the corresponding sensor variable is turned on. This constraint can be implemented as a SOS Type-1 constraint in \textsc{Gurobi} or \textsc{CPLEX}. It should be noted that the number of binary variables is $mT$ and the number of continuous variables is $\sum_{k=1}^T nmk = nm\frac{T(T+1)}{2}$. For high-dimensional systems and long horizons, this can quickly become intractable. It would be desirable to have a minimal formulation purely in terms of binary variables $\gamma$, which we leave to future work.

    \begin{rem}[Comparison with \cite{mo2011sensor}] \label{rem:comparison}
    At first glance, Theorem~\ref{thm:miqp} seems similar to the quadratic programming formulation in \cite{mo2011sensor}. However, the key difference is the relationship between the continuous variables and binary variables. In Theorem~\ref{thm:miqp}, there is one indicator constraint for each entry of the matrix $K_k$ which is easy to implement in modern solvers and also leads to fast solve times as evidenced in Section~\ref{section:results}. In \cite{mo2011sensor}, the constraint is more involved. It has the form $g(X) \leq \gamma_i$, where $\gamma_i$ is a binary decision variable and $g$ is the $\ell_0$ norm of the sum of $\ell_1$ norms of the columns of the matrix of continuous variables $X$. This constraint cannot be directly implemented in modern mixed integer solvers. One has to setup indicator variables to model the relationship between $\gamma_i$ and the $\ell_0$ norm. Unfortunately, this does not yield fast solve times as we will show in Section~\ref{section:results}. It should be noted that \cite{mo2011sensor} did not set out to solve an integer program; it relaxed the formulation to compute approximate solutions using a creative iterative weighting technique which yielded good solutions in simulations.
    \end{rem}

    \subsection{Covariance Matrices} \label{subsec:covariance-matrices}
    The last remaining piece is to specify the covariance matrices used in the definition of the functions $c_k$ in \eqref{eqn:lem2-defn}. To do this, we require a representation of any state $x_j$ in terms of the initial state $x_1$ which will make it easier to compute the required covariance matrices.
        \begin{observation}[State Representation] \label{observation:state-rep}
        For any $j \in \mathbb{N}$,
            \begin{equation}
                x_j = A^{j-1} x_1 + \sum_{i=1}^{j-1} A^{j-i-1} w_i,
            \end{equation}
        \end{observation}
        
        \begin{observation}[Zero-mean States] \label{observation:zero-mean-state}
        For any $j \in \mathbb{N}$, $\mathbb{E}(x_j) = 0$. This follows from the combination of Observation \ref{observation:state-rep} and the fact that the initial state $x_0$ and process noise $w_t$ are zero-mean random variables.
        \end{observation}

        Using these two observations, we will establish the structure of the covariance between any two states $x_i$ and $x_j$. 
        
        \begin{observation}[Covariance between States]
        For $j \leq i$, the covariance between states $x_i$ and $x_j$ is given by
            \begin{equation}
                \begin{split}
                    \Sigma_{x_i x_j} &= \mathbb{E} \Bigg[ \Big( A^{i-1} x_1 + \sum_{t=1}^{i-1} A^{i-t-1} w_t \Big) \\
                    &\Big( A^{j-1} x_1 + \sum_{t=1}^{j-1} A^{j-t-1} w_t \Big)'\Bigg] \in \mathbb{R}^{n \times n}\\
                    &= \mathbb{E} \Big( A^{i-1} x_1 x_1' A^{j-1'} + \sum_{t=1}^{j-1} A^{i-t-1} w_t w_t' A^{j-t-1'}\Big)\\
                    &= A^{i-1} \Sigma_{1|0} A^{j-1'} + \sum_{t=1}^{j-1} A^{i-t-1} W (A^{j-t-1})'.
                \end{split}
            \end{equation}
        For $j > i$, $\Sigma_{x_i x_j} = \Sigma_{x_j x_i}'$.
        \end{observation}

        Since the measurement is a linear combination of the state \eqref{eqn:dynamics}, we can determine the covariance between any state and measurement variable and also between measurement variables. This is described in the following observations.
        \begin{observation}[Covariance between State and Measurement] \label{obs:state-meas-cov}
        For $j \leq t$, the matrix of covariances between state $x_t \in \mathbb{R}^n$ and all sensor observation $y_j \in \mathbb{R}^m$ is given by
            \begin{equation}
                \begin{split}
                    \Sigma_{x_t y_j} &= \mathbb{E} \Big( x_t (C x_j + v_j)' \Big)\\
                    &= \Sigma_{x_t x_j} C' \in \mathbb{R}^{n \times m}
                \end{split}
            \end{equation}
        \end{observation}
        
        \begin{observation}[Covariance between Measurements] \label{obs:measurement-cov}
        The covariance between measurements $y_i$ and $y_j$ is given by
            \begin{equation}
                \begin{split}
                    \Sigma_{y_i y_j} &= \mathbb{E} \Big( (C x_i + v_i) (C x_j + v_j)' \Big)\\
                    &= C\ \Sigma_{x_i x_j} C' + \mathbb{E} (v_i v_j') \in \mathbb{R}^{m \times m}
                \end{split}
            \end{equation}
            If $i \neq j$ then $\mathbb{E} (v_i v_j') = 0$ (Kalman filter assumptions) and is equal to $V \in \mathbb{S}^m_{++}$ (noise covariance) otherwise.
        \end{observation}
        
        We now establish the form of the matrices in Theorem~\ref{thm:miqp} whose entries are given by the observations above.
            \begin{equation} \label{eqn:matrix-defns}
                \begin{split}
                    \Sigma_{x_k Y_k} &= \begin{bmatrix}
                    \Sigma_{x_k y_1} & \hdots & \Sigma_{x_k, y_k}
                    \end{bmatrix} \in \mathbb{R}^{n \times mk}, \\
                    \Sigma_{Y_k Y_k} &= \begin{bmatrix}
                    \Sigma_{y_1 y_1} & \hdots & \Sigma_{y_1 y_k}\\
                    \vdots & \ddots & \vdots\\
                    \Sigma_{y_k y_1} & \hdots & \Sigma_{y_k y_k} \\
                    \end{bmatrix} \in \mathbb{S}^{mk}_{++}.
                \end{split}
            \end{equation}

By substituting \eqref{eqn:matrix-defns} into \eqref{eqn:lem2-defn}, the MIQP provided in Theorem~\ref{thm:miqp} is now fully specified. This forms our approach to solving the general sensor scheduling problem.

\section{Evaluation} \label{section:results}
To evaluate our proposed MIQP, we consider two variants involving budget constraints: sensor selection (Section~\ref{res:selection}) and sensor scheduling (Section~\ref{res:scheduling}). For both problems, we compare against 
\begin{itemize}
    \item Mo \textit{et al}. \cite{mo2011sensor}: We convert the quadratic program using an $\ell_0$ constraint to a MIQP using an indicator constraint (see Section~\ref{rem:comparison} for a detailed discussion).
    \item Greedy algorithm: We keep selecting sensors one at a time until the budget constraint is met. In each step, the sensor yielding the highest reduction in error is chosen. For more details, see \cite{jawaid2015submodularity,chamon2020approximate}.
\end{itemize}
We now describe the system parameters. The process matrix $A \in \mathbb{R}^{n \times n}$ is drawn from a uniform distribution on $[0,1]$ and normalized so that the magnitude of the largest eigenvalue is 0.5. The measurement matrix $C \in \mathbb{R}^{m \times n}$ is drawn from a uniform distribution on $[0,1]$. The measurement noise matrix is constructed as $W = LL' \in \mathbb{S}^{n}_{+}$ where $L \in \mathbb{R}^{n \times n}$ is drawn from a uniform distribution on $[0,1]$. Finally, the noise matrix is $V = \sigma^2 \mathbb{I}_{m}$ where $\sigma^2 = 0.01$ and the initial state covariance is $\Sigma_{1|0} = \mathbb{I}_n$. We run each algorithm on 20 trials with each trial drawing a new set of system parameters. The experiments are implemented in \textsc{Gurobi} \cite{gurobi} on an AMD Ryzen 7 2700 8-core processor with 16 GB of RAM.

For both selection and scheduling, the goal is to minimize the final state estimation error. Thus, the cost matrices $Q_k \in \mathbb{R}^{n \times n}$ are defined as $Q_k = 0$ for $k \in [T-1]$ and $Q_T = \mathbb{I}_n$. The definitions of the constraint matrix and vector are given in Section~\ref{section:problem_formulation}. We now proceed to discuss the results.
\begin{figure}
    \centering
    \includegraphics[width=0.80\linewidth]{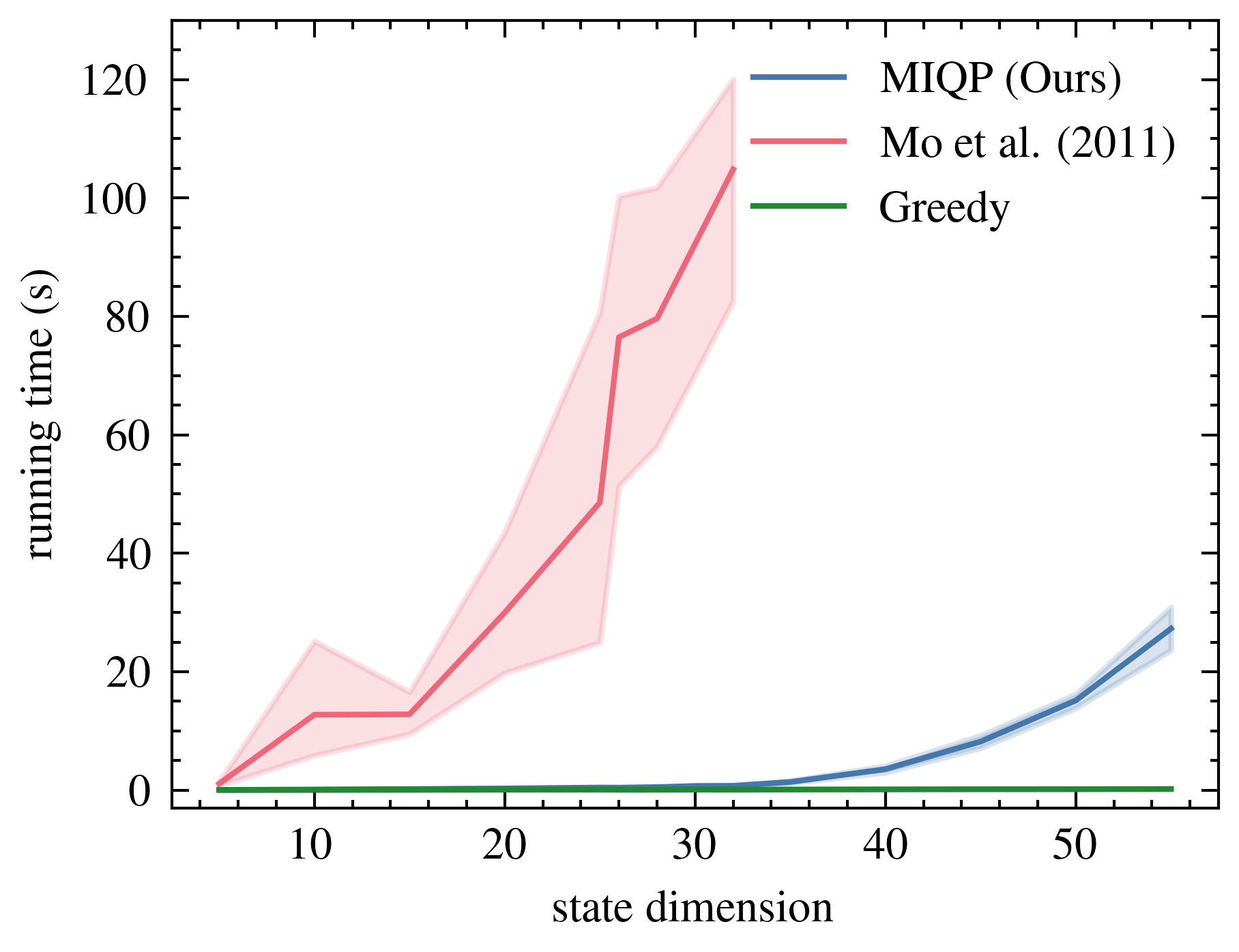}
    \caption{\textbf{Sensor Selection}: Runtime comparison of our MIQP, \cite{mo2011sensor}, and the greedy algorithm. The parameters are $T=3$ with $m=10$ sensors and a budget of $p=5$. The $x-$axis indicates the state dimension while the $y-$axis is the running time measured in seconds.}
    \label{fig:runtime-sensor-placement}
\end{figure}

\subsection{Sensor Selection with Budget Constraints} \label{res:selection}

The first question we seek to answer is how the runtime of the MIQP scales as the system size increases. We set a timeout of 120 seconds for each algorithm. The results of this experiment are shown in Figure~\ref{fig:runtime-sensor-placement}. We see that the approach of \cite{mo2011sensor} does not scale favourably and it times out on systems with $n > 30$. On the other hand, the proposed MIQP solves these instances in under a second which is on average, roughly 50 times faster for a system with 25 states. However, the proposed MIQP also shows an increase in runtime as the system size increases albeit at a much slower rate. In fact, until systems of size $n=35$, the runtime is comparable with greedy. The runtime of greedy is extremely efficient solving all instances in under a second. Further, the approach of \cite{mo2011sensor} exhibits high variance in its runtime which is in contrast to the proposed MIQP and the greedy algorithm which exhibit low runtime variability. In summary, the MIQP has a better runtime when compared to \cite{mo2011sensor} with timings up to 40 times faster while providing \emph{optimal} solutions.

\begin{figure}
    \centering
    \includegraphics[width=0.7\linewidth]{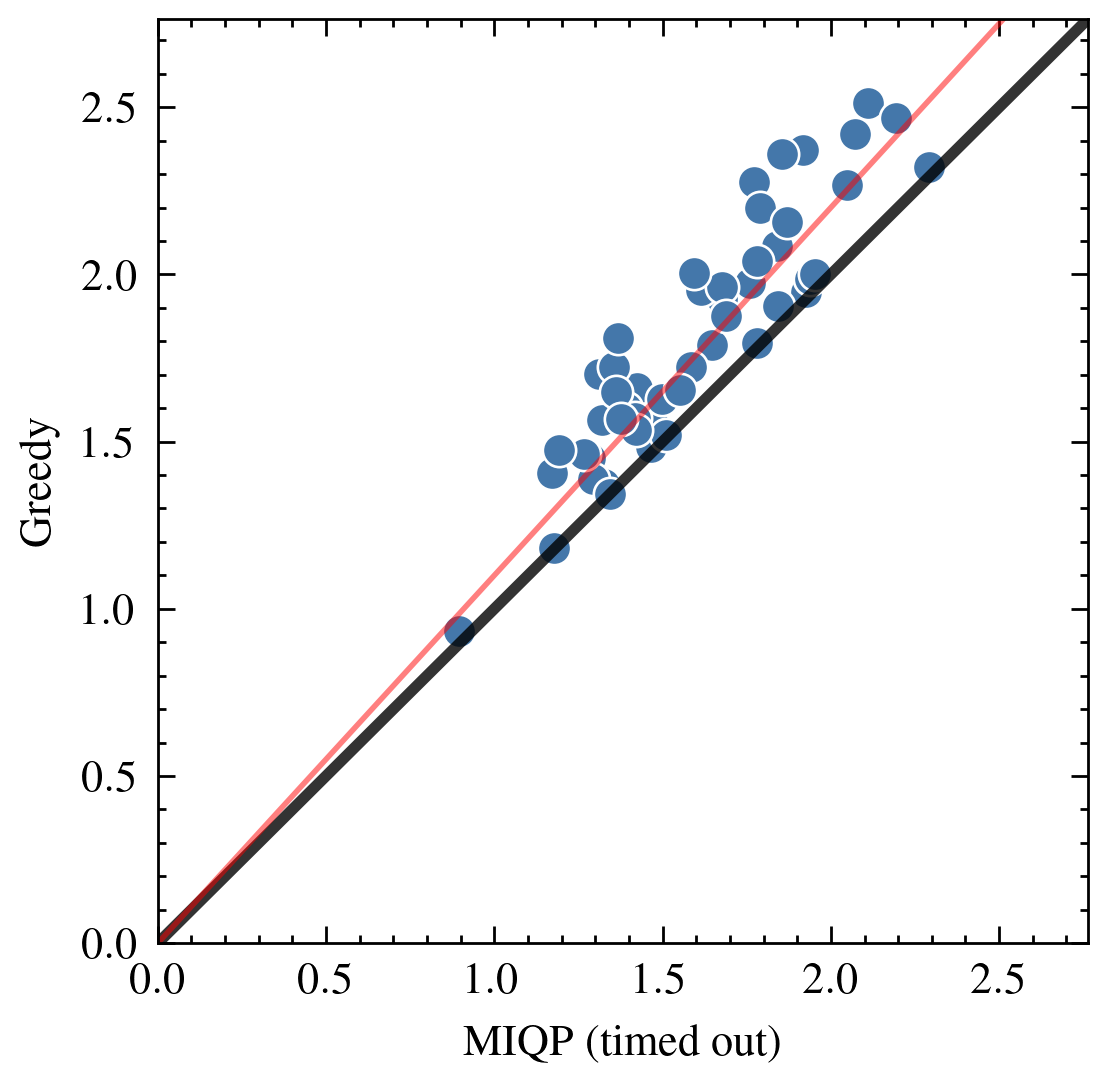}
    \caption{\textbf{Sensor Selection}: Solution quality comparison of our MIQP (timed out after 10 seconds) and the greedy algorithm. The $x$ and $y$ coordinate denote the MIQP (timed out) solution and the greedy solution respectively.}
    \label{fig:timeout-placement}
\end{figure}

We also investigate the quality of the solution returned by the MIQP when given a time out of 10 seconds. In this situation, it would be apt to compare it to greedy since both return approximate solutions. The setup for this experiment involves fixing the state dimension $n=10$, horizon $T=3$, number of sensors $m=25$, and budget $p=5$. We select a time out of 10 seconds as it is practical assumption on the runtime for an algorithm. We draw 50 random instances of the system parameters and run both algorithms to get their solutions. We observed that the MIQP timed out on \emph{all} drawn instances. For each instance $i$, we plot $(x_i,y_i)$ where $x_i$ is the objective value of the solution returned by the MIQP and $y_i$ is the objective value of the greedy solution on instance $i$. We plot these points in Figure~\ref{fig:timeout-placement}. We see that all points lie above the black line $y=x$ indicating that even in this challenging setting, the MIQP outperforms greedy on all instances even though it times out. The red line is the function $y = 1.1x$ and thus points that lie above it are instances where the greedy obtains solutions that are at least 10\% worse than the MIQP solution. We recorded the mean error difference to be 13\%. To summarize, these results suggest the 10 second timed out MIQP returns solutions of quality better or equal to that of the greedy algorithm. The added benefit is that the MIQP gives a certificate of suboptimality even though it times out whereas the greedy algorithm, in general, cannot.

\subsection{Sensor Scheduling with Budget Constraints} \label{res:scheduling}
We now study the runtime of each approach for sensor scheduling problems. 
 The results are shown in Figure~\ref{fig:runtime-sensor-scheduling}. Note that sensor scheduling is a more challenging problem than sensor selection as one has to select a sensor subset for each time step. We see the approach of \cite{mo2011sensor} can only tackle systems of size up to $n=14$ before timing out (set to 100 seconds). In contrast, our MIQP solves the same instances in under a second giving improvements of up to 40 to 60 times. The greedy remains extremely efficient and provides solutions in under a second for all instances. The summary here is that our MIQP can tackle systems of size up to $n=35$ while providing optimal solutions before starting to display an increase in runtime, thereby demonstrating its effectiveness.

Our final experiment compares the solution quality of the MIQP, when given a time out of 10 seconds, to the greedy algorithm. This follows the same system setup used in the second experiment for sensor selection (Section~\ref{res:selection}) which generates 50 random system instances. In this setting, the MIQP was observed to time out on \emph{all} drawn instances. The results are shown in Figure~\ref{fig:timeout-schedule}. The observations are similar to that in sensor selection. We note the MIQP always returns solutions whose quality is better or equal to that of the greedy algorithm. In addition, the points above the red line are the instances where the greedy solution is at least 10\% worse. When compared to the results in sensor selection (Figure~\ref{fig:timeout-placement}), the number of points above the red line is greater in Figure~\ref{fig:timeout-schedule}. This indicates the performance difference between greedy and the MIQP is larger in sensor scheduling which is also evidenced by the average error difference recorded to be $19\%$.

\begin{figure}
    \centering
    \includegraphics[width=0.8\linewidth]{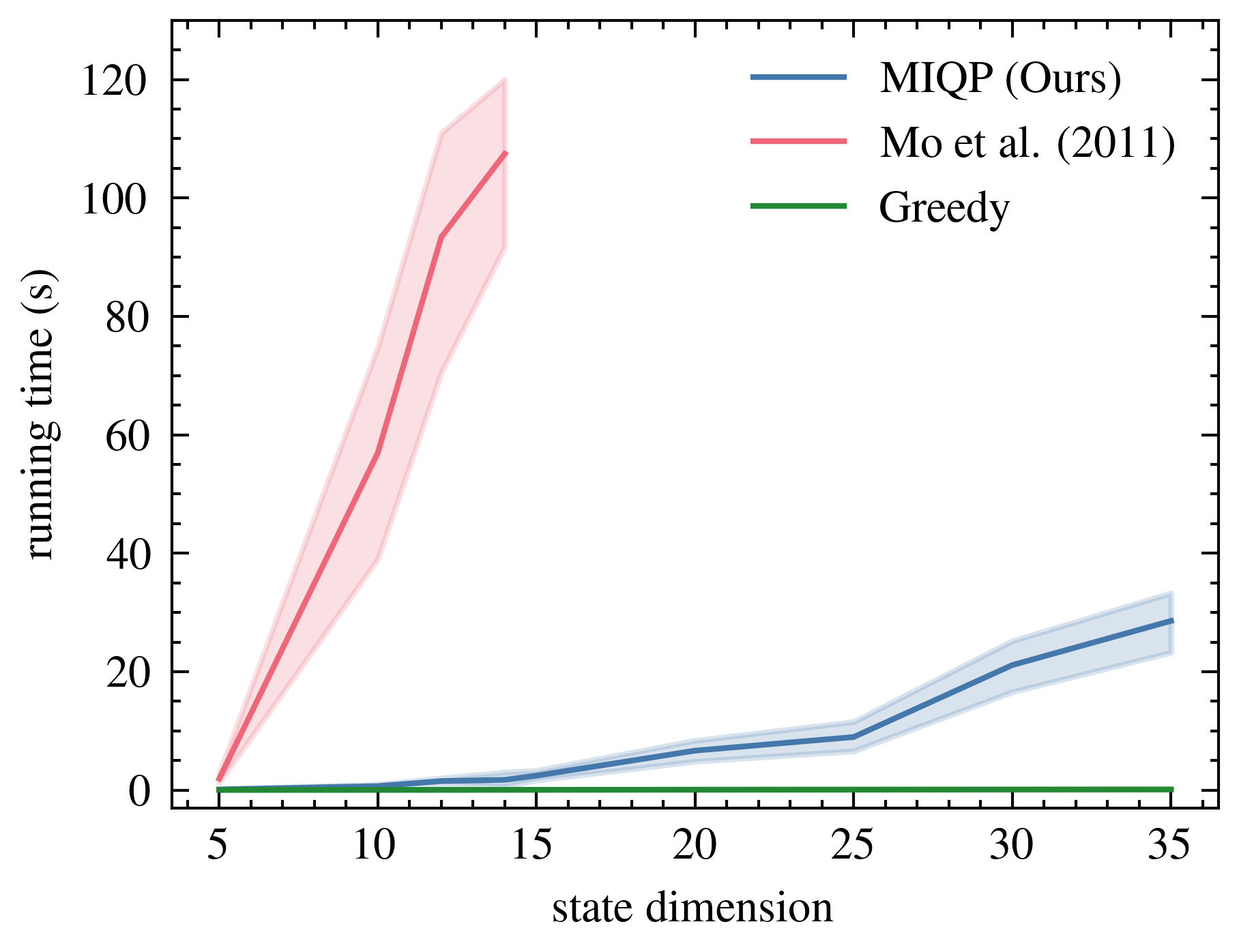}
    \caption{\textbf{Sensor Scheduling}: Runtime comparison of different algorithms on varying system sizes. The system parameters are $T=3$, $m=10$, and budget $p=5$. The runtime is measured in seconds indicated on the $y$-axis. }
    \label{fig:runtime-sensor-scheduling}
\end{figure}

\begin{figure}
    \centering
    \includegraphics[width=0.7\linewidth]{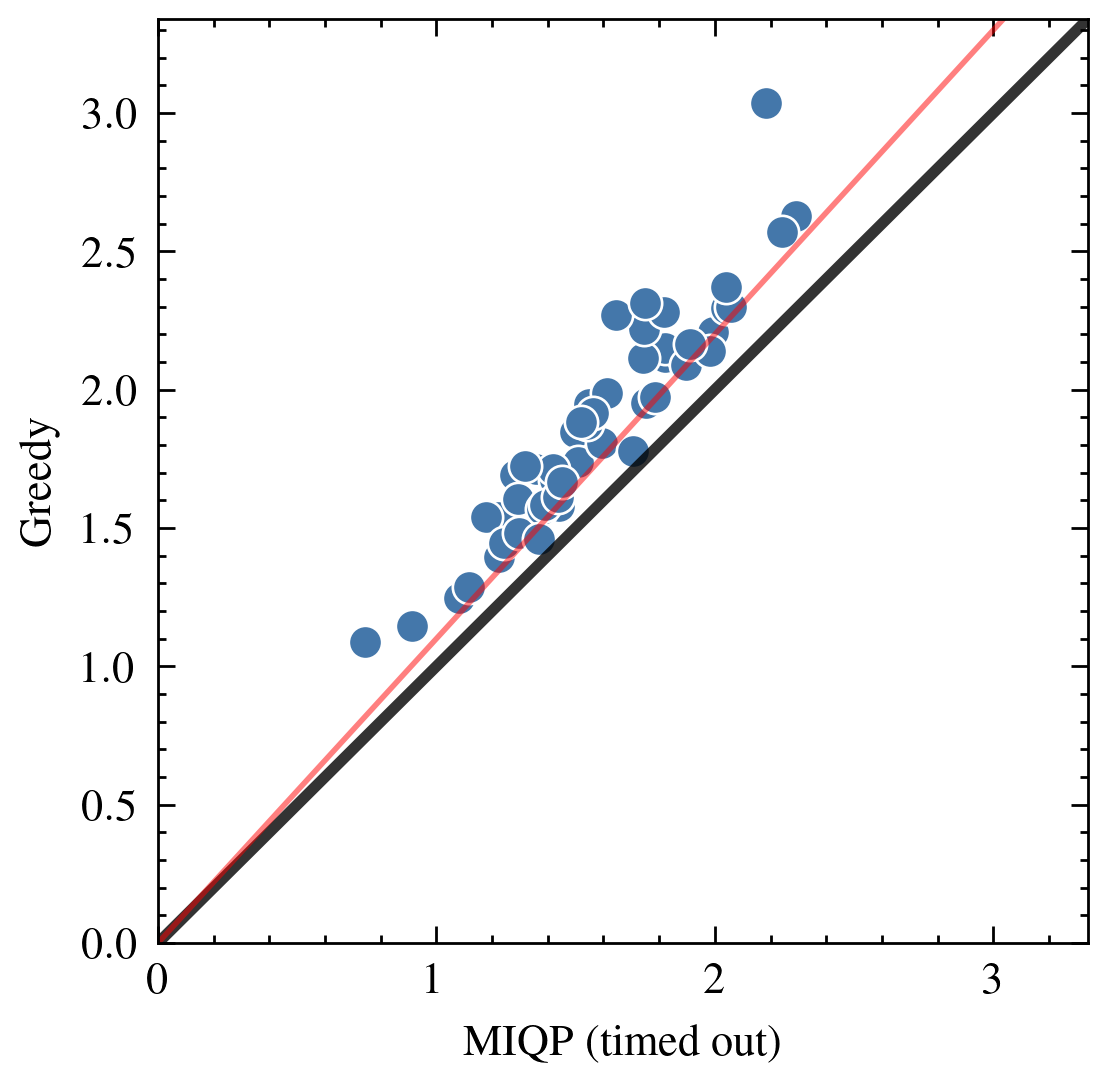}
    \caption{\textbf{Sensor Scheduling}: Comparison of solutions the MIQP (timed out after 10 seconds) and the greedy algorithm. The points above the black line indicate MIQP obtaining better solutions than greedy. The greedy solution of points above the red line are at least 10\% worse than the MIQP solution.}
    \label{fig:timeout-schedule}
\end{figure}
\section{Conclusions} \label{section:conclusions}

We studied the generalized version of the sensor scheduling problem capturing problems such as sensor placement, scheduling, and LQG sensing design. Our approach was rooted in mixed integer optimization. We formulated a mixed integer quadratic program by exploiting the optimality of the Kalman filter. In simulations, we showed the effectiveness of the approach in computing optimal solutions to systems with 30 to 50 states. In addition, the solver also returned better quality solutions over the popular greedy algorithm when constrained to time out within a few seconds. Looking forward, we believe a minimal formulation that solely contains integer variables would be desirable. This is because the number of continuous variables in the MIQP scales quadratically with the time horizon whereas the number of integer variables only depends on the number of sensors. Further, it would also be interesting to study the structure of the relaxed problem and identify any avenues for the development of approximation algorithms.

% ----------------------------------

\Urlmuskip=0mu plus 1mu
\bibliographystyle{IEEEtran}
\bibliography{mybib}
\end{document}